\documentclass{amsart}

\usepackage{amscd}
\usepackage{mathrsfs}
\usepackage{}
\usepackage[all]{xy}
\usepackage[OT1]{fontenc}
\usepackage[latin1]{inputenc}
\usepackage{amssymb,latexsym}
\usepackage{type1cm,ae}
\usepackage{graphicx}
\usepackage{xspace}
\usepackage{setspace}
\usepackage[english]{babel}
\newtheorem{theorem}{Theorem}[section]
\newtheorem{lemma}[theorem]{Lemma}

\theoremstyle{definition}

\newtheorem{example}[theorem]{Example}
\newtheorem{proposition}[theorem]{Proposition}

\DeclareMathOperator{\dist}{dist}
\DeclareMathOperator{\diam}{diam}

\newcommand*{\bR}{\ensuremath{\mathbb{R}}}

\newcommand*{\bdary}[1]{\partial #1}

\newcommand\numberthis{\addtocounter{equation}{1}\tag{\theequation}}

\def\Xint#1{\mathchoice
  {\XXint\displaystyle\textstyle{#1}}%
  {\XXint\textstyle\scriptstyle{#1}}%
  {\XXint\scriptstyle\scriptscriptstyle{#1}}%
  {\XXint\scriptscriptstyle\scriptscriptstyle{#1}}%
  \!\int}
\def\XXint#1#2#3{{\setbox0=\hbox{$#1{#2#3}{\int}$}
  \vcenter{\hbox{$#2#3$}}\kern-.5\wd0}}

\def\dashint{\Xint-}

\theoremstyle{remark}
\newtheorem{remark}[theorem]{Remark}

\numberwithin{equation}{section}

\begin{document}

\title{Fractional Sobolev-Poincar\'e inequalities in irregular domains}

\author{Chang-Yu Guo}
\address[Chang-Yu Guo]{Department of Mathematics and Statistics, University of Jyv\"askyl\"a, P.O. Box 35, FI-40014 University of Jyv\"askyl\"a, Finland}
\email{changyu.c.guo@jyu.fi}
\thanks{C.Y.Guo was supported by the Magnus Ehrnrooth foundation.}






\begin{abstract}
This paper is devoted to the study of fractional $(q,p)$-Sobolev-Poincar\'e inequalities in irregular domains. In particular, we  establish (essentially) sharp fractional $(q,p)$-Sobolev-Poincar\'e inequality in $s$-John domains and in domains satisfying the quasihyperbolic boundary conditions. When the order of the fractional derivative tends to 1, our results tends to the results for the usual derivative. Furthermore, we verified that those domains that support the fractional$(q,p)$-Sobolev-Poincar\'e inequality together with a separation property are $s$-diam John domains for certain $s$, depending only on the associated data. We also point out an inaccurate statement in~\cite{bk95}.
\end{abstract}

\maketitle

\section*{Introduction}\label{sec:intro}
Recall that a bounded domain $\Omega\subset \bR^n$ is a John domain if there
is a constant $C$ and a point $x_0\in \Omega$ so that, for each $x\in \Omega,$
one can find a rectifiable curve $\gamma:[0,1]\to \Omega$ with $\gamma(0)=x,$
$\gamma(1)=x_0$ and with
\begin{equation} \label{eka}
    Cd(\gamma(t),\bdary\Omega)\ge l(\gamma([0,t]))
\end{equation}
for each $0<t\le 1.$ F. John used this condition in his work on
elasticity~\cite{j61} and the term was coined by
Martio and Sarvas~\cite{ms79}. Smith and Stegenga \cite{ss90} introduced
the more general concept of $s$-John domains, $s\ge 1,$ by replacing \eqref{eka}
with
\begin{equation} \label{toka}
    Cd(\gamma(t),\bdary\Omega)\ge l(\gamma([0,t]))^s.
\end{equation}
The condition~\ref{eka} is called a ``twisted cone condition" in literature. Thus condition~\ref{toka} should be called a ``twisted cusp condition". 

In the last twenty years, $s$-John domains has been extensively studied in connection with Sobolev type inequalities; see~\cite{bk95,hk98, hk00, km00,kot02, ss90}. In particular, Buckley and Koskela~\cite{bk95} have shown that a simply connected planar domain which supports a Sobolev-Poincar\'e inequality is an $s$-John domain for an appropriate $s$. Smith and Stegenga have shown that an $s$-John domain $\Omega$ is a $p$-Poincar\'e domain, provided $s<\frac{n}{n-1}+\frac{p-1}{n}$. In particular, if $s<\frac{n}{n-1}$, then $\Omega$ is a $p$-Poincar\'e domain for all $1\leq p<\infty$. These results were further generalized to the case of $(q,p)$-Poincar\'e domains in~\cite{hk98, km00, kot02}. Recall that a bounded domain $\Omega\subset \bR^n$, $n\geq 2$, is said to be a $(q,p)$-Poincar\'e domain if there exists a constant $C_{q,p}=C_{q,p}(\Omega)$ such that
\begin{equation}\label{def:poincare domain}
\Big(\int_\Omega |u(x)-u_\Omega|^q dx\Big)^{1/q}\leq C_{q,p}\Big(\int_\Omega |\nabla u(x)|^p dx\Big)^{1/p}
\end{equation} 
for all $u\in C^\infty(\Omega)$. Here $u_\Omega=\dashint_{\Omega}u(x)dx$. When $q=p$, $\Omega$ is termed a $p$-Poincar\'e domain and when $q>p$ we say that $\Omega$ supports a Sobolev-Poincar\'e inequality.

In this paper, we consider the following fractional $(q,p)$-Sobolev-Poincar\'e inequality in a domain $\Omega\subset\bR^n$ with finite Lebesgue measure, $n\geq 2$:
\begin{align}\label{eq:fractional Sob Poin}
\int_\Omega|u(x)-u_\Omega|^qdx\leq C\Big(\int_\Omega\int_{\Omega\cap B(x,\tau d(x,\bdary\Omega))}\frac{|u(x)-u(y)|^p}{|x-y|^{n+p\delta}}dy dx \Big)^{q/p},
\end{align}
where $1\leq p\leq q<\infty$, $\delta\in (0,1)$, $\tau \in (0,\infty)$ and the constant $C$ does not depend on $u\in C(\Omega)$. If $\Omega$ supports the fractional $(q,p)$-Sobolev-Poincar\'e inequality~\eqref{eq:fractional Sob Poin}, $q\geq p$, then we say that $\Omega$ is a fractional $(q,p)$-Sobolev-Poincar\'e domain. 

From now on, unless specified, $\delta\in (0,1)$ and $\tau \in (0,\infty)$ will be fixed constants. Given a function $u\in C(\Omega)$, we define $g_u:\Omega\to \bR$ as 
\begin{align}\label{eq:def g u}
g_u(x)=\int_{\Omega\cap B(x,\tau d(x,\bdary\Omega))}\frac{|u(x)-u(y)|^p}{|x-y|^{n+p\delta}}dy
\end{align}
for $x\in \Omega$.

\begin{theorem}\label{thm:equivalence of Sob poi via capacity}
Let $\Omega\subset\bR^n$, $n\geq 2$, be a domain with finite Lebesgue measure and $1\leq p\leq q<\infty$. Then the following statements are equivalent:
\begin{itemize}
\item[i)] $\Omega$ satisfies the fractional $(q,p)$-Sobolev-Poincar\'e inequality;

\item[ii)] For an arbitrary ball $B_0\subset \Omega$ there exists a constant $C=C(\Omega,p,q,B_0)$ such that
\begin{align}\label{eq:capacity required}
|A|^{p/q}\leq C\inf \int_\Omega g_u(x)dx
\end{align}
for every measurable set $A\subset \Omega$ such that $A\cap B_0=\emptyset$. The infimum above is taken over all functions $u\in C(\Omega)$ that satisfy $u|_A\geq 1$ and $u|_{B_0}=0$.
\end{itemize}

\end{theorem}

\begin{theorem}\label{thm:main thm}
Let $\Omega\subset \bR^n$, $n\geq 2$, be an $s$-John domain. If $p<n/\delta$, $s<\frac{n}{n-p\delta}$ and $1\leq p\leq q<\frac{np}{s(n-p\delta)+(s-1)(p-1)}$, then $\Omega$ supports the fractional $(q,p)$-Sobolev-Poincar\'e inequality~\eqref{eq:fractional Sob Poin}.
\end{theorem}

%
The range for $q$ in Theorem~\ref{thm:main thm} is essentially sharp as indicated by the following example.
\begin{example}\label{example:sharp}
Given $\tau,\delta\in (0,1)$, $1\leq p<n/\delta$ and $s<\frac{n}{n-p\delta}$, there exists an $s$-John domain $\Omega\subset\bR^n$ such that $\Omega$ does not support any fractional $(q,p)$-Sobolev-Poincar\'e inequality with $q>\frac{np}{s(n-p\delta)+(s-1)(p-1)}$.
\end{example}

Theorem~\ref{thm:main thm} holds for the critical case $q=\frac{np}{s(n-p\delta)+(s-1)(p-1)}$ as well, provided $s=1$ or $p=1$; see Remark~\ref{rmk:rmk on the case p=1}. We conjecture that Theorem~\ref{thm:main thm} holds, under the same assumptions, for the critical case $q=\frac{np}{s(n-p\delta)+(s-1)(p-1)}$.

\begin{theorem}\label{thm:quasihyperbolic bc case}
Let $\Omega\subset\bR^n$, $n\geq 2$, satisfy the quasihyperbolic boundary condition~\eqref{eq:quasihyperbolic boundary condition} for some $\beta\leq 1$. Then $\Omega$ is a fractional $(q,p)$-Sobolev-Poincar\'e domain provided $p\in (\frac{1}{\delta}(n-n\frac{2\beta}{1+\beta}),n)$ and $q\in [p,\frac{2\beta}{1+\beta}\frac{np}{n-p\delta})$.
\end{theorem}

\begin{example}\label{example:qhbc}
For each $q>\frac{2\beta}{1+\beta}\frac{np}{n-p\delta}$, there exists a domain $\Omega\subset\bR^n$, $n\geq 2$, satisfying~\eqref{eq:quasihyperbolic boundary condition} which is not a fractional $(q,p)$-Sobolev-Poincar\'e domain. For each $1\leq p<\frac{1}{\delta}(n-n\frac{2\beta}{1+\beta})$, there exist domains $\Omega\subset \bR^n$, $n\geq 2$, satisfying~\eqref{eq:quasihyperbolic boundary condition} which is not a fractional $(p,p)$-Sobolev-Poincar\'e domain.
\end{example}

Recall that we say a domain $\Omega\subset \bR^n$ with a distinguished point $x_0$ has a separation property if there exists a constant $C_0$ such that the following property holds: for every $x\in \Omega$, there exists a curve $\gamma:[0,1]\to \Omega$ with $\gamma(0)=x$, $\gamma(1)=x_0$, and such that for each $t$ either 
\begin{align*}
\gamma([0,t])\subset B_t:=B(\gamma(t),C_0d(\gamma(t),\bdary\Omega))
\end{align*}
or each $y\in \gamma([0,t])\backslash B_t$ and $x_0$ belong to different components of $\Omega\backslash \bdary B_t$.

\begin{theorem}\label{thm:necessarility of s John}
Assume that $\Omega\subset \bR^n$ is a domain of finite Lebesgue measure that  satisfies the separation property with a distinguished point $x_0$. Let $1\leq p<\frac{n}{\delta}$. If $\Omega$ is a fractional $(q,p)$-Sobolev-Poincar\'e domain with $\tau=1$ for some $q>p$, 
then for each $x\in \Omega$, there is curve $\gamma:[0,1]\to \Omega$ with $\gamma(0)=x$, $\gamma(1)=x_0$ such that
\begin{align}\label{eq:s diam curve}
\diam \gamma([0,t])\leq C\varphi(d(\gamma(t),\bdary\Omega)),
\end{align}
where $\varphi(t)=t^{\frac{(n-p\delta)q}{p\delta}(\frac{1}{p}-\frac{1}{q})}$.
\end{theorem}

The assumptions in Theorem~\ref{thm:necessarility of s John} can be further relaxed. Indeed, Theorem~\ref{thm:necessarility of s John} holds if we only assume that the fractional $(q,p)$-Sobolev-Poincar\'e inequality~\eqref{eq:fractional Sob Poin} holds for all locally Lipschitz continuous functions in $\Omega$; see Remark~\ref{rmk:on the function spaces}.

Since the paper generalizes the main results of~\cite{hk98,kot02,div14,bk95} to the fractional setting in a natural way, some of the arguments used in this paper are similar to ones in those papers. In particular, we benefit a lot from~\cite{hk98} and~\cite{kot02}. This paper is organized as follows. Section 2 contains the basic definitions and
Section 3 some auxiliary results. We prove our main results, namely Theorem~\ref{thm:equivalence of Sob poi via capacity}, Theorem~\ref{thm:main thm} and Example~\ref{example:sharp}, in Section 4.
In Section 5, we prove Theorem~\ref{thm:quasihyperbolic bc case} and give the construction of Example~\ref{example:qhbc}. In the final section, Section 6, we discuss the proof of Theorem~\ref{thm:necessarility of s John} and point out a mistake in~\cite{bk95}.

\section{Notations and definitions}

Recall that the quasihyperbolic metric $k_\Omega$ in a domain $\Omega\subsetneq\bR^n$ is defined to be
\begin{equation*}
    k_\Omega(x,y)=\inf_\gamma k_\Omega-\text{length}(\gamma),
\end{equation*}
where the infimum is taken over all rectifiable curves $\gamma$ in $\Omega$ which join $x$ to $y$ and
\begin{equation*}
    k_\Omega-\text{length}(\gamma)=\int_\gamma \frac{ds}{d(x,\bdary\Omega)}
\end{equation*}
denotes the quasihyperbolic length of $\gamma$ in $\Omega$. This metric was introduced by Gehring and Palka in~\cite{gp76}. A curve $\gamma$ joining $x$ to $y$ for which $k_\Omega$-length$(\gamma)=k_\Omega(x,y)$ is called a quasihyperbolic geodesic. Quasihyperbolic geodesics joining any two points of a proper subdomain of $\bR^n$ always exists; see~\cite[Lemma 1]{go79}. 

Let $\Omega$ be a bounded domain in $\bR^n$, $n\geq 2$. Then $\mathcal{W}=\mathcal{W}(\Omega)$ denotes a Whitney decomposition of $\Omega$, i.e. a collection of closed cubes $Q\subset\Omega$ with pairwise disjoint interiors and having edges parallel to the coordinate axes, such that $\Omega=\cup_{Q\in \mathcal{W}}Q$, the diameters of $Q\in \mathcal{W}$ belong to the set $\{2^{-j}:j\in \mathbb{Z}\}$ and satisfy the condition
\begin{equation*}
\diam(Q)\leq \dist(Q,\bdary\Omega)\leq 4\diam(Q).
\end{equation*}
For $j\in\mathbb{Z}$ we define
\begin{equation*}
\mathcal{W}_j=\{Q\in \mathbb{W}:\diam(Q)=2^{-j}\}.
\end{equation*}

Note that when we write $f(x)\lesssim g(x)$, we mean that $f(x)\leq Cg(x)$ is satisfied for all $x$ with some fixed constant $C\geq 1$. Similarly, the expression $f(x)\gtrsim g(x)$ means that $f(x)\geq C^{-1}g(x)$ is satisfied for all $x$ with some fixed constant $C\geq 1$. We write $f(x)\approx g(x)$ whenever $f(x)\lesssim g(x)$ and $f(x)\gtrsim g(x)$.
\section{Auxiliary results}\label{sec:auxiliary results}

We need the following ``chain lemma" from~\cite[Proof of Theorem 9]{hak98}.

\begin{lemma}\label{lemma:chain}
Let $\Omega\subset \bR^n$ be an $s$-John domain and $M>1$ a fixed constant. Let $B_0=B(x_0,\frac{d(x_0,\bdary\Omega)}{4M})$, where $x_0\in \Omega$ is the John center. There exists a constant $c>0$, depending only on $\Omega$, $M$ and $n$, such that given $x\in \Omega$, there exists a finite ``chain" of balls $B_i=B(x_i,r_i)$, $i=0,1,\cdots,k$ ($k$ depends on the choice of $x$) that joins $x_0$ to $x$ with the following properties:

\begin{itemize}
 \item[1.] $|B_i\cup B_{i+1}|\leq c|B_i\cap B_{i+1}|$;
 \item[2.] $d(x,B_i)\leq cr_i^{1/s}$;
 \item[3.] $d(B_i,\bdary\Omega)\geq Mr_i$;
 \item[4.] $\sum_{i=0}^k\chi_{B_i}\leq c\chi_\Omega$;
 \item[5.] $|x-x_i|\leq cr_i^{1/s}$ and $B_k=B(x,\frac{d(x,\bdary\Omega)}{4M})$;
 \item[6.] For any $r>0$, the number of balls $B_i$ with radius $r_i>r$ is less than $cr^{(1-s)/s}$ when $s>1$.
\end{itemize}
\end{lemma}

Recall that for a function $f$, the Riesz potential $I_\delta$, $\delta\in (0,n)$, of $f$ is defined by
\begin{align}\label{eq:definition of Riesz potential}
I_\delta(f)=\int_{\bR^n}\frac{f(y)}{|x-y|^{n-\delta}}dy.
\end{align}

The following estimate for Riesz potential is well-known; see for instance~\cite[Theorem 3.1.4 and Corollary 3.1.5]{ah96}.
\begin{theorem}\label{thm:Riesz potential estimate}
Let $0<\delta<n$, $1<p<q<\infty$, and $1/p-1/q=\delta/n$. Then $\|I_\delta(f)\|_q\leq c\|f\|_p$ for some constant $c$ independent of $f\in L^p(\bR^n)$. Moreover, there is a constant $c_1=c(n,\delta)>0$ such that the weak estimate
\begin{align}\label{eq:weak type estimate for Riesz}
\sup_{t>0}|\{x\in \bR^n:|I_\delta(f)(x)|>t\}|t^{n/(n-\delta)}\leq c_1\|f\|_1^{n/(n-\delta)}
\end{align}
holds for every $f\in L^1(\bR^n)$.
\end{theorem}

The following proposition, which can regarded as a fractional analogy of~\cite[Theorem 2.1]{bk95},  is proved in~\cite[Proposition 6.2]{div14}.
\begin{proposition}\label{prop:for necessarility of s John}
Suppose that $\Omega\subset \bR^n$ is a domain of finite Lebesgue measure. Let $1\leq p<q<\infty$. Assume that the fractional $(q,p)$-Sobolev-Poincar\'e inequality~\eqref{eq:fractional Sob Poin} holds with $\tau=1$ for every $u\in C(\Omega)$. Fix a ball $B_0\subset\Omega$, and let $d>0$ and $w\in \Omega$. Then there exists a constant $C>0$ such that
\begin{align*}
\diam(T)\leq C\Big(d+|T|^{(\frac{1}{p}-\frac{1}{q})\frac{1}{\delta}}\Big)
\end{align*}
and
\begin{align*}
|T|^{1/n}\leq C(d+d^{(n-p\delta)q/(np)})
\end{align*}
if $T$ is the union of all components of $\Omega\backslash B(w,d)$ that do not intersect the ball $B_0$. The constant $C$ depends only on $|B_0|,|\Omega|,n,p,q,\delta$ and the constant associated to the fractional $(q,p)$-Sobolev-Poincar\'e inequality.
\end{proposition}
\begin{remark}\label{rmk:on the function spaces}
As in~\cite{bk95}, one can check that the conclusion holds whenever the fractional $(q,p)$-Sobolev-Poincar\'e inequality~\eqref{eq:fractional Sob Poin} with $\tau=1$ holds for every locally Lipschitz continuous functions; see~\cite[Proof of Proposition 6.2]{div14}. 
\end{remark}

The following lemma is proved in~\cite[Lemma 2.6]{kot02}.
\begin{lemma}\label{lemma:key lemma}
Let $\Omega\subset\bR^n$, $n\geq 2$, be a domain that satisfies the quasihyperbolic boundary condition~~\eqref{eq:quasihyperbolic boundary condition}. Then for each $\varepsilon>0$ there exists a constant $C=C(n,\diam\Omega,\varepsilon)$ such that
\begin{align}\label{eq:27}
\sup_{Q_1\in \mathcal{W}}\sum_{Q\in P(Q_1)}|Q|^\varepsilon\leq C.
\end{align}
\end{lemma}

Fix a Whitney cube $Q_0$ and assume that $x_0$ is the center of $Q_0$. For each cube $Q\in \mathcal{W}$, we choose a quasihyperbolic geodesic $\gamma$ joining $x_0$ to the center of $Q$ and we let $P(Q)$ denote the collection of all the Whitney cubes $Q'\in \mathcal{W}$ which intersect $\gamma$. Then the shadow $S(Q)$ of the cube $Q$ is defined to be
\begin{align*}
S(Q)=\bigcup_{Q_1\in \mathcal{W},Q\in P(Q_1)}Q_1.
\end{align*}

We need the following estimate of the size of the shadow of a Whitney cube $Q$ in terms of the size of $Q$. The proof is essentially contained in~\cite[Lemma 2.8]{kot02} with minor modifications.
\begin{lemma}\label{lemma:bound diam of shadow}
Let $\Omega\subset\bR^n$, $n\geq 2$, be a domain that satisfies the quasihyperbolic boundary condition~~\eqref{eq:quasihyperbolic boundary condition}. Then there exists a constant $C=C(n)$ such that
\begin{align*}
\diam S(Q)\leq C (\diam Q)^{\frac{2\beta}{1+\beta}}
\end{align*}
for all $Q\in \mathcal{W}$. Consequently,
\begin{align}\label{eq:29}
|S(Q)|\leq C|Q|^{\frac{2\beta}{1+\beta}}.
\end{align}
\end{lemma}

\section{Main proofs}\label{sec:main proofs}
\begin{proof}[Proof of Theorem~\ref{thm:equivalence of Sob poi via capacity}]
We first show that condition $ii$ implies condition $i$. Fix a function $u\in C(\Omega)$. Pick a real number $b$ such that both $|\{x\in \Omega:u(x)\geq b\}|$ and $|\{x\in \Omega:u(x)\leq b\}|$ are at least $|\Omega|/2$. It suffices to show the fractional $(q,p)$-Sobolev-Poincar\'e inequality with $|u-u_\Omega|$ replaced by $|u-b|$, and by replacing $u$ with $u-b$, we may assume that $b=0$. Write $v_{+}=\max\{u,0\}$ and $v_{-}=-\min\{u,0\}$. In the sequel $v$ denotes either $v_{+}$ or $v_{-}$; all the statements below are valid in both cases. Without loss of generality, we may assume that $v\geq 0$.

For each $j\in \mathbb{Z}$, we define $v_j(x)=\min\{2^j,\max\{0,v(x)-2^j\}\}$. We next prove the following inequality
\begin{align}\label{eq:sufficient ineq}
2^{qj}|\{x\in \Omega:v_j(x)\geq 2^j\}|\leq C\Big(\int_\Omega g_{v_j}(x)dx\Big)^{q/p}.
\end{align}
To see it, notice that $2^{-j}v_j|_{B_0}=0$ and $2^{-j}v_j|_{F_j}\geq 1$, where $F_j=\{x\in \Omega:v_j(x)\geq 2^j\}$. So by~\eqref{eq:capacity required}, we obtain that
\begin{align*}
|F_j|^{p/q}\leq C\int_\Omega g_{2^{-j}v_j}(x)dx.
\end{align*}
Note that $g_{2^{-j}v_j}=2^{-pj}g_{v_j}$. Thus we finally arrive at
\begin{align*}
2^{pj}|F_j|^{p/q}\leq C\int_\Omega g_{v_j}(x)dx,
\end{align*}
which is the desired estimate~\eqref{eq:sufficient ineq}.

The fractional $(q,p)$-Sobolev-Poincar\'e inequality now follows from the weak type estimates via a standard argument. Write $B_{y}=B(y,\tau d(y,\bdary\Omega))$ and $A_k=F_{k-1}\backslash F_{k}$.
\begin{align*}
\int_\Omega|v(x)|^qdx&\leq \sum_{k\in \mathbb{Z}}2^{(k+1)q}|A_k|\leq C\sum_{k\in \mathbb{Z}}\Big(\int_\Omega g_{v_k}(x)dx\Big)^{q/p}\\
&\leq C\Big(\sum_{k\in \mathbb{Z}}\int_\Omega g_{v_k}(x)dx\Big)^{q/p}\\
&\leq C\Big(\sum_{k\in \mathbb{Z}}(I_1^{k}+I_2^{k})\Big)^{q/p},
\end{align*}
where 
\begin{align*}
I_1^k=\sum_{i\leq k+1}\sum_{j\geq k+1}\int_{A_i}\int_{A_j\cap B_y}\frac{|v_k(y)-v_k(z)|^p}{|y-z|^{n+p\delta}}dzdy
\end{align*}
and
\begin{align*}
I_2^k=\sum_{i\geq k+1}\sum_{j\leq k+1}\int_{A_i}\int_{A_j\cap B_y}\frac{|v_k(y)-v_k(z)|^p}{|y-z|^{n+p\delta}}dzdy.
\end{align*}
For $y\in A_i$ and $z\in A_j$ with $j-1>i$, $|v(y)-v(z)|\geq |v(z)|-|v(y)|\geq 2^{j-2}$. Hence,
\begin{align}\label{eq:easy consequence}
|v_k(y)-v_{k}(z)|\leq 2^{k+1}\leq 4\cdot 2^{k+1-j}|v(y)-v(z)|.
\end{align}
Since the estimate
\begin{align*}
|v_k(y)-v_k(z)|\leq |v(y)-v(z)|
\end{align*}
holds for every $k\in \mathbb{Z}$, ~\eqref{eq:easy consequence} is valid whenever $i\leq k\leq j$ and $(y,z)\in A_i\times A_j$. It follows from~\eqref{eq:easy consequence} that
\begin{align*}
\sum_{k\in \mathbb{Z}} I_1^k\leq 4^p\sum_{k\in\mathbb{Z}}\sum_{i\leq k+1}\sum_{j\geq k+1}2^{p(k+1-j)}\int_{A_i}\int_{A_j\cap B_y}\frac{|v(y)-v(z)|^p}{|y-z|^{n+p\delta}}dzdy.
\end{align*}
Since $\sum_{k=i-1}^{j-1} 2^{p(k+1-j)}\leq (1-2^{-p})^{-1}$, changing the order of the summation yields that the right hand side in the above inequality is bounded by
\begin{align*}
\frac{4^p}{1-2^{-p}}\int_\Omega g_v(y)dy.
\end{align*}
The estimate of $I_2^k$ is similar. Thus, we have proved that
\begin{align*}
\int_\Omega|v(x)|^qdx\leq C\Big(\int_\Omega g_v(y)dy\Big)^{q/p}.
\end{align*}
The desired fractional $(q,p)$-Sobolev-Poincar\'e inequality~\eqref{eq:fractional Sob Poin} follows from the above inequality by noticing that $|u|=v_{+}+v_{-}$ and that $|v_{\pm}(y)-v_{\pm}(z)|\leq |u(y)-u(z)|$ for all $y,z\in\Omega$.

The implication from condition $ii$ to condition $i$ is easier. To see it, fix a measurable set $A\subset \Omega$ such that $A\cap B_0=\emptyset$ and a function $u\in C(\Omega)$ such that $u|_A\geq 1$ and $u|_{B_0}=0$. If $u_\Omega\leq \frac{1}{2}$, then by~\eqref{eq:fractional Sob Poin} we have
\begin{align*}
2^{-q}|A|&\leq \int_A |u(x)-u_\Omega|^qdx\leq \int_\Omega |u(x)-u_\Omega|^qdx\\
&\leq C\Big(\int_\Omega g_u(y)dy\Big)^{q/p}.
\end{align*}
If $u_\Omega\geq \frac{1}{2}$, then by~\eqref{eq:fractional Sob Poin} we have
\begin{align*}
2^{-q}|A|&\leq 2^{-q}\frac{|\Omega|}{|B_0|}|B_0|\leq \frac{|\Omega|}{|B_0|}\int_{B_0}|u(x)-u_\Omega|^qdx\\
&\leq \frac{|\Omega|}{|B_0|}C\Big(\int_\Omega g_u(y)dy\Big)^{q/p}.
\end{align*}
Combining the above two estimates, we conclude that
\begin{align*}
|A|^{p/q}\leq C\int_\Omega g_u(x)dx,
\end{align*}
where $C=C(\Omega,B_0,p,q)$. Taking the infimum over all such $u$ gives us~\eqref{eq:capacity required}.

\end{proof}

\begin{proof}[Proof of Theorem~\ref{thm:main thm}]
Let $B_0=B(x_0,\frac{d(x_0,\bdary\Omega)}{4M})$. Assume that $p<n/\delta$, $1<s<\frac{n}{n-p\delta}$ and $1\leq p\leq q<\frac{np}{s(n-p\delta)+(s-1)(p-1)}$. Choose $\Delta>0$ such that
$$2\Delta=\frac{np}{q}-s(n-p\delta)-(s-1)(p-1).$$
It suffices to show, by Theorem~\ref{thm:equivalence of Sob poi via capacity}, that there exists a constant $C=C(\Omega,p,q,B_0)$ such that for every measurable set $A\subset \Omega$ with $A\cap B_0=\emptyset$, we have
\begin{align*}
|A|^{p/q}\leq C\inf \int_\Omega g_u(x)dx
\end{align*}
whenever $u\in C(\Omega)$ satisfies $u|_A\geq 1$ and $u|_{B_0}=0$. 

For any $x\in A$, we obtain from Lemma~\ref{lemma:chain} a finite chain of balls $B_i$, $i=0,1,\cdots,k$, satisfying conditions 1-6 with $M>2/\tau$.  For all $i=0,1,\cdots,k$, we have
\begin{equation}\label{eq:boundary estimate}
B_i\subset B(y,\tau d(y,\bdary\Omega)), \quad \text{if}\ y\in B_i.
\end{equation}
To see this, fix $y\in B_i$ and let $z$ be any other point in $B_i$, then by condition 3 in Lemma~\ref{lemma:chain},
\begin{align*}
|z-y|&\leq |y-x_i|+|x_i-z|\leq 2r_i\leq 2\frac{d(B_i,\bdary\Omega)}{M}\\
&\leq \frac{2}{M}d(y,\bdary\Omega)<\tau d(y,\bdary\Omega).
\end{align*}

In order to estimate $|A|$, we divide $A$ into the ``bad" and ``good" parts. Set 
\begin{align*}
\mathscr{G}=\Big\{x\in A|u_{B_x}\geq \frac{1}{2}\Big\}\quad \text{and}\quad \mathscr{B}=A\backslash \mathscr{G}.
\end{align*}
We have $|A|\leq |\mathscr{G}|+|\mathscr{B}|$ and we first estimate $|\mathscr{G}|$.

By condition 1 in Lemma~\ref{lemma:chain}, we have
\begin{align*}
\frac{1}{2}&\leq |u_{B_k}-u_{B_0}|\leq \sum_{i=0}^{k-1}|u_{B_i}-u_{B_{i+1}}|\\
&\leq \sum_{i=0}^{k-1}\Big(|u_{B_i}-u_{B_i\cap B_{i+1}}|+ |u_{B_{i+1}}-u_{B_i\cap B_{i+1}}|\Big)\\
&\lesssim \sum_{i=0}^k \frac{1}{|B_i|}\int_{B_i}|u(y)-u_{B_i}|dy.
\end{align*}

For a ball $B_i$,
\begin{align*}
\frac{1}{|B_i|}&\int_{B_i}|u(y)-u_{B_i}|dy \leq \frac{1}{|B_i|}\int_{B_i}\Big(\frac{1}{|B_i|}\int_{B_i}|u(y)-u(z)|^pdz \Big)^{1/p}dy\\
&=\frac{1}{|B_i|^{1+1/p}}\int_{B_i}\Big(\int_{B_i}|u(y)-u(z)|^pdz \Big)^{1/p}dy\\
&\lesssim |B_i|^{\delta/n-1}\int_{B_i}\Big(\int_{B_i}\frac{|u(y)-u(z)|^p}{|y-z|^{n+p\delta}}dz \Big)^{1/p}dy
\end{align*}
Set 
\begin{align*}
g(y):=\Big(\int_{\Omega\cap B(y,\tau d(y,\bdary\Omega))}\frac{|u(y)-u(z)|^p}{|y-z|^{n+p\delta}}dz \Big)^{1/p}
\end{align*}
By~\eqref{eq:boundary estimate} and condition 2 in Lemma~\ref{lemma:chain},
\begin{align*}
\sum_{i=0}^k \frac{1}{|B_i|}&\int_{B_i}|u(y)-u_{B_i}|dy\\
&\lesssim \sum_{i=0}^k|B_i|^{\delta/n-1}\int_{B_i}\Big(\int_{B_i}\frac{|u(y)-u(z)|^p}{|y-z|^{n+p\delta}}dz \Big)^{1/p}dy\\
&\leq \sum_{i=0}^k|B_i|^{\delta/n-1}\int_{B_i}\Big(\int_{B(y,\tau d(y,\bdary\Omega))}\frac{|u(y)-u(z)|^p}{|y-z|^{n+p\delta}}dz \Big)^{1/p}dy\\
&\lesssim \sum_{i=0}^k r_i^{\delta-n/p}\Big(\int_{B_i}g(y)^pdy\Big)^{1/p}.
\end{align*}
Thus we conclude that
\begin{align*}
1\lesssim\sum_{i=0}^k r_i^{\delta-n/p}\Big(\int_{B_i}g(y)^pdy\Big)^{1/p}.
\end{align*}
H\"older's inequality implies
\begin{equation*}
1\lesssim \Big(\sum_{i=0}^kr_i^{\kappa p/(p-1)} \Big)^{(p-1)/p}\Big(\sum_{i=0}^kr_i^{p(-\kappa+\delta-n/p )}\int_{B_i}g(y)^pdy \Big)^{1/p},
\end{equation*}
where $\kappa=\frac{(s-1)(p-1)+\Delta}{sp}$. Using condition 6 from Lemma~\ref{lemma:chain}, one can easily conclude
\begin{equation*}
\sum_{i=0}^kr_i^{\kappa p/(p-1)}\leq \sum_{i=0}^\infty (2^{-i})^{\kappa p/(p-1)}2^{i(s-1)/s}<C.
\end{equation*}
Therefore,
\begin{equation}\label{eq:hk23}
\sum_{i=0}^kr_i^{p(-\kappa+\delta-n/p )}\int_{B_i}g(y)^pdy \geq C,
\end{equation}
where the constant $C$ depends only on $p$, $n$, $\Delta$ and the constant from $s$-John condition.

By condition 2 from Lemma~\ref{lemma:chain}, $Cr_i\geq |x-y|^s$, for $y\in B_i$, and since $p(-\kappa+\delta-n/p )<0$ according to our choice $p\leq n/\delta$, we obtain
\begin{equation*}
r_i^{-\kappa p-n+\delta}\lesssim |x-y|^{s(-\kappa p-n+p\delta)}
\end{equation*}
for $y\in B_i$. For $y\in B_i\cap (2^{j+1}B_k\backslash 2^{j}B_k)$, we have $|x-y|\approx 2^jr_k$ and hence for such $y$,
\begin{equation}\label{eq:hk24}
r_i^{-\kappa p-n+p\delta}\lesssim (2^jr_k)^{s(-\kappa p-n+p\delta)}.
\end{equation}
Combining~\eqref{eq:hk23} with~\eqref{eq:hk24} leads to
\begin{align*}
1&\lesssim \sum_{i=0}^kr_i^{p(-\kappa+\delta-n/p )}\int_{B_i}g(y)^pdy\lesssim (r_k)^{s(-\kappa p-n+p\delta)}\int_{B_i}g(y)^pdy\\
&+\sum_{j=0}^{|\log r_k|}(2^jr_k)^{s(-\kappa p-n+p\delta)}\int_{(2^{j+1}B_k\backslash 2^{j}B_k)\cap\Omega}g(y)^pdy\\
&\lesssim \sum_{l=0}^{|\log r_k|+1}(2^lr_k)^{s(-\kappa p-n+p\delta)}\int_{2^lB_k\cap\Omega} g(y)^pdy.
\end{align*}
On the other hand,
\begin{equation*}
\sum_{l=0}^{|\log r_k|+1}(2^lr_k)^\Delta <r_k^\Delta\sum_{l=-\infty}^{|\log r_k|+1}2^{l\Delta}<C.
\end{equation*}
Comparing the above two estimates, we conclude that there exists an $l$ (depending on $\Delta$) such that
\begin{equation*}
(2^lr_k)^\Delta\lesssim (2^lr_k)^{s(-\kappa p-n+p\delta)}\int_{2^lB_k\cap\Omega}g(y)^pdy.
\end{equation*}
It follows that,
\begin{equation*}
\int_{\Omega\cap 2^lB_k}g(y)^pdy\gtrsim (2^lr_k)^{s(n+\kappa p-p\delta)+\Delta}=(2^lr_k)^{s(n-p\delta)+(s-1)(p-1)+2\Delta}.
\end{equation*}
In other words, there exists an $R_x\geq d(x,\bdary\Omega)/2$ with
\begin{equation*}
\Big(\int_{\Omega\cap B(x,R_x)}g(y)^pdy\Big)^{\frac{np}{q[s(n-p\delta)+(s-1)(p-1)+2\Delta]}}\gtrsim (R_x^n)^{p/q}.
\end{equation*}
Note that according to our choice of $\Delta$, the above estimate reduces to the following form:
\begin{align*}
\int_{\Omega\cap B(x,R_x)}g(y)^pdy\gtrsim |B(x,R_x)|^{p/q}.
\end{align*}
Applying the Vitali covering lemma to the covering $\{B(x,R_x)\}_{x\in E}$ of the set $\mathscr{B}$, we can select pairwise disjoint balls $B_1,\dots,B_k,\dots$ such that $\mathscr{B}\subset \bigcup_{i=1}^\infty 5B_i$. Let $r_i$ denote the radius of the ball $B_i$. Then
\begin{align*}
|\mathscr{G}|&\leq \sum_{i=1}^\infty |5B_i|=5^n\sum_{i=1}^\infty |B_i|
\lesssim \sum_{i=1}^\infty\Big(\int_{\Omega\cap B_i}g_u(y)dy\Big)^{\frac{q}{p}}\\
&\lesssim \Big(\sum_{i=1}^\infty\int_{\Omega\cap B_i}g_u(y)dy\Big)^{\frac{q}{p}}
\lesssim \Big(\int_{\Omega}g_u(y)dy\Big)^{\frac{q}{p}}.
\end{align*}

We next estimate $|\mathscr{B}|$. Note that $\mathscr{B}\subset \bigcup_{x\in \mathscr{B}}B_x$. We may use the Besicovitch covering theorem to select a subcovering $\{B_{x_i}\}_{i\in \mathbb{N}}$. Since $u\geq 1$ on $A$, and $u_{B_{x_i}}\leq 1/2$, we obtain that 
\begin{align*}
|u(y)-u_{B_{x_i}}|^q\geq 2^{-q}
\end{align*}
for $y\in A\cap B_{x_i}$. By the fractional $(q,p)$-Sobolev-Poincar\'e inequality for balls, we get
\begin{align*}
|A\cap B_{x_i}|&\leq C\int_{A\cap B_{x_i}}|u(y)-u_{B_{x_i}}|^qdy\\
&\leq C\Big(\int_{B_{x_i}}g_u(y)dy\Big)^{q/p}.
\end{align*}
Summing over all balls $B_{x_i}$, we obtain that
\begin{align*}
|\mathscr{B}|^{p/q}\leq C\int_{\Omega}g_u(y)dy.
\end{align*}
The proof of Theorem~\ref{thm:main thm} is now complete.

\end{proof}

\begin{remark}\label{rmk:rmk on the case p=1}
In Theorem~\ref{thm:main thm}, $q$ is assumed to be strictly less than $\frac{np}{s(n-p\delta)+(s-1)(p-1)}$. However, one can easily adapt the proof of Theorem~\ref{thm:main thm} to show that when $s=1$ or $p=1$, $q$ can reach the critical value. Indeed, we only need to use a variant of Lemma~\ref{lemma:chain}. Namely, for each $x\in \Omega$, we may join $x$ to $x_0$ via a infinite chain of balls $\{B_i\}_{i\in \mathbb{N}}$ with all the properties listed in Lemma~\ref{lemma:chain} except condition 5 replaced with
\begin{align*}
|x-x_i|\leq cr_i^{1/s}\to 0 
\end{align*}
as $i\to \infty$. Then following the proof of Theorem~\ref{thm:main thm}, we easily deduce the following Riesz potential type estimate:
\begin{align*}
|u(x)-u_{B_0}|\lesssim \sum_{i=1}^\infty r_i^{\delta-n}\int_{B_i}g(y)dy\lesssim \int_\Omega\frac{g(y)}{|x-y|^{s(n-\delta)}}dy.
\end{align*}
Note that 
\begin{align*}
\int_\Omega\frac{g(y)}{|x-y|^{s(n-\delta)}}dy=I_\delta(\chi_\Omega g)(x).
\end{align*}
Thus we conclude that
\begin{align*}
|u(x)-u_{B_0}|\lesssim I_\delta(\chi_\Omega g)(x).
\end{align*}
For $s=1$ and $p>1$, the claim follows from the strong type estimate in Theorem~\ref{thm:Riesz potential estimate}. For $p=1$, the claim follows from the weak type estimate~\eqref{eq:weak type estimate for Riesz}.
\end{remark}

\begin{proof}[Proof of Example~\ref{example:sharp}]
We will use the mushroom-like domain as used in~\cite{hak98}. The mushroom-like domain $\Omega\subset\bR^n$ consists of a cube $Q$ and an attached infinite sequences of mushrooms $F_1,F_2,\cdots$ growing on the ``top" of the cube. By a mushroom $F$ of size $r$, we mean a cap $\mathscr{C}$, which is a ball of radius $r$, and an attached cylindrical stem $\mathscr{P}$ of height $r$ and radius $r^s$. The mushrooms are disjoint, and the corresponding cylinders are perpendicular to the side of the cube that we have selected as the top of the cube. We can make the mushrooms pairwise disjoint if the number $r_i$ associated with $F_i$ converges to 0 sufficiently fast as $i\to \infty$.

Let $u_i$ be a piecewise linear function on $\Omega$ such that $u_i=0$ outside $F_i$, $u_i=1$ on the cap $\mathscr{C}_i$, and $u_i$ is linear on the associated cylinder $\mathscr{P}_i$. Assume that $1\leq s<\frac{n}{n-p\delta}$, and that one can prove the fractional $(q,p)$-Sobolev-Poincar\'e inequality with $q>\frac{np}{s(n-p\delta)+(s-1)(p-1)}$.

Note that 
\begin{align*}
\Big(\int_\Omega|u(x)-u_\Omega|^qdx\Big)^{1/q}\gtrsim r_i^{n/q}.
\end{align*}
On the other hand,
\begin{align*}
\Big(\int_\Omega&\int_{\Omega\cap B(x,\tau d(x,\bdary\Omega))}\frac{|u(x)-u(y)|^p}{|x-y|^{n+p\delta}}dx \Big)^{1/p}\\
&=\Big(\int_{\mathscr{P}_i}\int_{\mathscr{P}_i\cap B(x,\tau d(x,\bdary\Omega))}\frac{|u(x)-u(y)|^p}{|x-y|^{n+p\delta}}dx \Big)^{1/p}\\
&\lesssim \Big(r_i^{-p}\int_{\mathscr{P}_i}d(x,\bdary\Omega)^{p(1-\delta)}dx \Big)^{1/p}\\
&\lesssim \Big(r_i^{s(n-p\delta)+(s-1)(p-1)}\Big)^{1/p}.\\
\end{align*}
Thus we obtain that for all $i\in \mathbb{N}$
\begin{align*}
r_i^{n/q}\lesssim r_i^{\frac{s(n-p\delta)+(s-1)(p-1)}{p}},
\end{align*}
which is impossible if $q>\frac{np}{s(n-p\delta)+(s-1)(p-1)}$.
\end{proof}

\section{Fractional $(q,p)$-Sobolev-Poincar\'e inequalities in domains with quasihyperbolic boundary condition}
Recall that a domain $\Omega\subset\bR^n$, $n\geq 2$, is said to satisfy a $\beta$-quasihyperbolic boundary condition, $\beta\in (0,1]$, if there exist a point $x_0\in \Omega$ and a constant $C_0$ such that 
\begin{equation}\label{eq:quasihyperbolic boundary condition}
k_\Omega(x,x_0)\leq \frac{1}{\beta}\log\frac{d(x_0,\bdary\Omega)}{d(x,\bdary\Omega)}+C_0
\end{equation}
holds for all $x\in \Omega$.


\begin{proof}[Proof of Theorem~\ref{thm:quasihyperbolic bc case}]
Fix $Q_0\subset\Omega$ the central Whitney cube containing $x_0$. For each measurable set $A\subset\Omega$ with $A\cap Q_0=\emptyset$, let $u\in C(\Omega)$ satisfy $u|_A\geq 1$ and $u|_{Q_0}=0$. As in the proof of Theorem~\ref{thm:main thm}, we divide $A$ into ``good" and ``bad" parts. Set 
\begin{align*}
\mathscr{G}=\Big\{x\in A|u_{Q}\geq \frac{1}{2}\quad \text{for some Whitney cube}\quad Q\ni x\Big\}\quad \text{and}\quad \mathscr{B}=A\backslash \mathscr{G}.
\end{align*}
We have $|A|\leq |\mathscr{G}|+|\mathscr{B}|$ and we first estimate $|\mathscr{B}|$.

For points $x\in \mathscr{B}$, the standard fractional $(p',p)$-Sobolev-Poincar\'e inequality on cubes provides a trivial estimate
\begin{align*}
|A\cap Q|^{1/p'}\leq C\Big(\int_Q|u-u_Q|^{p'}dy\Big)^{1/p'}\leq C\Big(\int_Qg_u(y)dy\Big)^{1/p}
\end{align*}
on Whitney cube $Q$ containing $x$. Since $q<p'$ this yields
\begin{align*}
\int_Qg_u(y)dy\geq \frac{1}{C}|A\cap Q|^{p/q}
\end{align*}
and by summing over all such Whitney cubes we deduce that 
\begin{align}\label{eq:34}
\int_\Omega g_u(y)dy\geq \frac{1}{C}|\mathscr{B}|^{p/q}.
\end{align}

We next estimate $|\mathscr{G}|$ and our aim is the show that
\begin{align}\label{eq:35}
\int_\Omega g_u(y)dy\geq \frac{1}{C}|\mathscr{G}|^{p/q}
\end{align}
and then the conclusion follows from Theorem~~\ref{thm:equivalence of Sob poi via capacity}.

For each $x\in \mathscr{G}$, let $Q(x)$ be the Whitney cube containing $x$ for which $u_{Q(x)}\geq \frac{1}{2}$. Then the chaining argument used in the proof of Theorem~\ref{thm:main thm} gives us the estimate
\begin{align}\label{eq:36}
1\lesssim \sum_{Q\in P(Q(x))}(\diam Q)^{\delta-n/p}\Big(\int_{Q}g_u(y)dy\Big)^{1/p};
\end{align}
recall that $P(Q(x))$ consists of the collection of all the Whitney cubes which intersect the quasihyperbolic geodesic joining $x_0$ to the center of $Q(x)$.

Integrating~\eqref{eq:36} with respect to the Lebesgue measure and interchanging the order of summation and integration yields
\begin{align*}
|\mathscr{G}|&\lesssim \int_{\mathscr{G}}\sum_{Q\in P(Q(x))}(\diam Q)^{\delta-n/p}\Big(\int_{Q}g_u(y)dy\Big)^{1/p}dx\\
&=\sum_{Q\in \mathcal{W}}|S(Q)\cap \mathscr{G}|(\diam Q)^{\delta-n/p}\Big(\int_{Q}g_u(y)dy\Big)^{1/p}.\numberthis\label{eq:37}
\end{align*}
Applying H\"older's inequality leads to
\begin{align*}
|\mathscr{G}|&\lesssim \Big(\sum_{Q\in \mathcal{W}}|S(Q)\cap \mathscr{G}|^{\frac{p}{p-1}}|Q|^{-\frac{n-p\delta}{n(p-1)}} \Big)^{\frac{p-1}{p}}\Big(\sum_{Q\in \mathcal{W}}\int_{Q}g_u(y)dy \Big)^{\frac{1}{p}}\\
&\leq\Big(\sum_{Q\in \mathcal{W}}|S(Q)\cap \mathscr{G}|^{\frac{p}{p-1}}|Q|^{-\frac{n-p\delta}{n(p-1)}} \Big)^{\frac{p-1}{p}}\Big(\int_{\Omega}g_u(y)dy \Big)^{\frac{1}{p}}.
\end{align*}
Applying Lemma~\ref{lemma:for proof of qhbc} below, we find that
\begin{align*}
|\mathscr{G}|\lesssim |\mathscr{G}|^{(q-1)/q}\Big(\int_{\Omega}g_u(y)dy \Big)^{\frac{1}{p}},
\end{align*}
which proves~\eqref{eq:35}.

\end{proof}

\begin{lemma}\label{lemma:for proof of qhbc}
Fix $p$ and $q$ as in Theorem~~\ref{thm:quasihyperbolic bc case}. Then there exists a constant $C=C(n,p,q,\beta)$ such that
\begin{align*}
\sum_{Q\in \mathcal{W}}|S(Q)\cap E|^{\frac{p}{p-1}}|Q|^{-\frac{n-p\delta}{n(p-1)}}\leq C|E|^{\frac{p}{p-1}\frac{q-1}{q}}
\end{align*}
whenever $E\subset\Omega$.
\end{lemma}
\begin{proof}
For simplicity, we write $p^*=\frac{np}{n-p\delta}$, $\kappa=\frac{p}{p-1}$ and $\lambda=\frac{q}{q-1}$. Then $\frac{n-p\delta}{n(p-1)}=\frac{\kappa}{p^*}$. Thus
\begin{align*}
\sum_{Q\in \mathcal{W}}&|S(Q)\cap E|^{\kappa}|Q|^{-\frac{\kappa}{p^*}}\leq |E|^{\frac{\kappa}{p}-\frac{\kappa}{q}}\sum_{Q\in \mathcal{W}}\sum_{Q_1\in S(Q)}|Q_1\cap E|\Big(\frac{|S(Q)|^{\frac{1}{q}}}{|Q|^{\frac{1}{p^*}}}\Big)^{\kappa}\\
&=|E|^{\frac{\kappa}{p}-\frac{\kappa}{q}}\sum_{Q_1\in\mathcal{W}}|Q_1\cap E|\sum_{Q\in P(Q_1)}\Big(\frac{|S(Q)|^{\frac{1}{q}}}{|Q|^{\frac{1}{p^*}}}\Big)^{\kappa}\\
&\lesssim |E|^{\frac{\kappa}{p}-\frac{\kappa}{q}}\sum_{Q_1\in\mathcal{W}}|Q_1\cap E|\sum_{Q\in P(Q_1)}|Q|^{(\frac{2\beta}{1+\beta}\frac{1}{q}-\frac{1}{p^*})\kappa}\\
&\lesssim |E|^{\frac{\kappa}{p}-\frac{\kappa}{q}}\sum_{Q_1\in\mathcal{W}}|Q_1\cap E|=|E|^{\frac{\kappa}{\lambda}},
\end{align*}
where we have used~\eqref{eq:29} and~\eqref{eq:27} with~$\varepsilon=(\frac{2\beta}{(1+\beta)q}-\frac{1}{p^*})\kappa>0$.

\end{proof}

\begin{proof}[Proof of Example~\ref{example:qhbc}]
The construction here is similar to that used in the proof of Example~\ref{example:sharp} and thus we only point out the difference.  The mushroom-like domain $\Omega\subset\bR^n$ consists of a cube $Q$ and an attached infinite sequences of mushrooms $F_1,F_2,\cdots$ growing on the ``top" of the cube as in Example~\ref{example:sharp}. Now, by a mushroom $F$ of size $r$, we mean a cap $\mathscr{C}$, which is a ball of radius $r$, and an attached cylindrical stem $\mathscr{P}$ of height $r^\tau$ and radius $r^\sigma$. The mushrooms are disjoint, and the corresponding cylinders are perpendicular to the side of the cube that we have selected as the top of the cube. We can make the mushrooms pairwise disjoint if the number $r_i$ associated with $F_i$ converges to 0 sufficiently fast as $i\to \infty$.

It is easy to show that $\Omega$ satisfies the $\beta$-quasihyperbolic boundary condition~\eqref{eq:quasihyperbolic boundary condition} if $\sigma=\frac{1+\beta}{2\beta}\leq \tau$; see for instance~\cite[Example 5.5]{kot02}. We next show that $\Omega$ is not a fractional $(q,p)$-Sobolev-Poincar\'e domain if 
\begin{align}\label{eq:56}
q>\frac{np}{\sigma(n-p\delta)+(p-1)(\sigma-\tau)}.
\end{align}
When $\tau=\sigma=\frac{1+\beta}{2\beta}$, ~\eqref{eq:56} implies that $\Omega$ is a $\beta$-quasihyperbolic boundary condition boundary which does not support a fractional $(q,p)$-Sobolev-Poincar\'e inequality. This verifies Example~\ref{example:qhbc}.

Let $u_i$ be a piecewise linear function on $\Omega$ such that $u_i=0$ outside $F_i$, $u_i=1$ on the cap $\mathscr{C}_i$, and $u_i$ is linear on the associated cylinder $\mathscr{P}_i$. Assume that the fractional $(q,p)$-Sobolev-Poincar\'e inequality holds on $\Omega$.

Note that 
\begin{align*}
\Big(\int_\Omega|u(x)-u_\Omega|^qdx\Big)^{1/q}\gtrsim r_i^{n/q}.
\end{align*}
On the other hand,
\begin{align*}
\Big(\int_\Omega&\int_{\Omega\cap B(x,\tau d(x,\bdary\Omega))}\frac{|u(x)-u(y)|^p}{|x-y|^{n+p\delta}}dx \Big)^{1/p}\\
&=\Big(\int_{\mathscr{P}_i}\int_{\mathscr{P}_i\cap B(x,\tau d(x,\bdary\Omega))}\frac{|u(x)-u(y)|^p}{|x-y|^{n+p\delta}}dx \Big)^{1/p}\\
&\lesssim \Big(r_i^{-\tau p}\int_{\mathscr{P}_i}d(x,\bdary\Omega)^{p(1-\delta)}dx \Big)^{1/p}\\
&\lesssim \Big(r_i^{\sigma(n-p\delta)+(p-1)(\sigma-\tau)}\Big)^{1/p}.\\
\end{align*}
Thus we obtain that for all $i\in \mathbb{N}$
\begin{align*}
r_i^{n/q}\lesssim r_i^{\frac{\sigma(n-p\delta)+(p-1)(\sigma-\tau)}{p}},
\end{align*}
which is impossible if $q>\frac{np}{\sigma(n-p\delta)+(\sigma-\tau)(p-1)}$.
\end{proof}

\section{Necessary conditions for the fractional $(q,p)$-Sobolev-Poincar\'e domains}

\begin{proof}[Proof of Theorem~\ref{thm:necessarility of s John}]
Fix $x\in \Omega$. Pick a curve $\gamma:[0,1]\to \Omega$ with $\gamma(0)=x$ and $\gamma(1)=x_0$ as in the definition of separation property.

Let $0<t<1$ and $\delta(t)=d(\gamma(t),C\delta(t))$, there is nothing to prove. Otherwise, the separation property implies that $\bdary B=\bdary B(\gamma(t),C\delta(t))$ separates $\gamma([0,t])\backslash B$ from $x_0$. If the component of $\Omega\backslash \bdary B$ containing $x_0$ does not contain a ball centred at $x_0$ of radius $\delta(1)/2$, then $B$ must have radius at least $\delta(1)/4$ since it intersects both $B(x_0,\delta(1)/2)$ and $\bdary\Omega$. In this case, $B'=4B$ contains $B(x_0,\delta(1)/4)$ and we may assume that $B'$ does not contain $\gamma([0,t])$ (since otherwise we are done). Thus either $\Omega\backslash \bdary B$ or $B'$ contains a ball centred at $x_0$ of radius comparable to $\delta(1)$. In either cases, we conclude from Proposition~\ref{prop:for necessarility of s John} that 
\begin{align*}
\diam \gamma([0,t])\leq C\varphi(d(\gamma(t),\bdary\Omega)),
\end{align*}
where $\varphi(t)=t^{\frac{(n-p\delta)q}{p\delta}(\frac{1}{p}-\frac{1}{q})}$.
\end{proof}

A bounded domain $\Omega\subset\bR^n$ with a distinguished point $x_0$ satisfying~\eqref{eq:s diam curve} with $\varphi(t)=t^{1/s}$ is termed $s$-diam John in~\cite{gk14}. It was proved in~\cite{gk14} that, for $s>1$, $s$-diam John domains are not necessarily $s$-John.

In~\cite[Corollary 4.1]{bk95}, it was stated that if a bounded domain $\Omega\subset\bR^n$ satisfies a separation property and supports a $(q,p)$-Sobolev-Poincar\'e inequality~\eqref{def:poincare domain} with $q>p$, then $\Omega$ is $s$-John with $s=\frac{p^2}{(n-p)(q-p)}$. One could immediately check that the proof given there was only sufficient to deduce that $\Omega$ is $s$-diam John with $s=\frac{p^2}{(n-p)(q-p)}$. In fact, combining~\cite[Example 5.1]{gk14} and~\cite[Section 4]{bk95}, one can produce an $s$-diam John domain $\Omega\subset\bR^n$ with $s=\frac{p^2}{(n-p)(q-p)}$ such that $\Omega$ supports a $(q,p)$-Sobolev-Poincar\'e inequality. Moreover, $\Omega$ is not $s'$-diam John whenever $s'<s$ and $\Omega$ is not $s$-John. 

We next briefly discuss how to construct such an example in the plane (it works in higher dimensions as well). Set 
\begin{align*}
C(r;\alpha, \beta)=C(r)=\{(x_1,x):0<x_1<r^\alpha,|x'|<r^\beta\},
\end{align*}
where $0<\alpha<\beta\leq 1$ will be specified later. The idea is very simple, we first use the mushroom-like domain $\Omega'\subset\bR^2$ constructed as in~\cite{bk95} (with different choices of parameters) and then modify $\Omega'$ to be a spiral domain $\Omega$ as in~\cite[Example 5.1]{gk14}. 

The mushroom-like domain $\Omega'\subset\bR^2$ consists of a cube $Q$ and an attached infinite sequences of mushrooms $F_1,F_2,\cdots$ growing on the ``top" of the cube as in Example~\ref{example:sharp}. Now, by a mushroom $F$ of size $r$, we mean a cap $\mathscr{C}$, which is a ball of radius $r$, and an attached cylindrical stem $C(r)$. The mushrooms are disjoint, and the corresponding cylinders are perpendicular to the side of the cube that we have selected as the top of the cube. We can make the mushrooms pairwise disjoint if the number $r_i$ associated with $F_i$ converges to 0 sufficiently fast as $i\to \infty$.

Note first that if $\beta=\alpha\frac{p+(p-1)q}{(n-1)(q-p)}$ with $n=2$, then $C(r)$ satisfies the $(q,p)$-Sobolev-Poincar\'e inequality uniformly in $r$; see~\cite{bk95}. Let $\mu=s\beta=\frac{p^2}{(2-p)(q-p)}\beta$ and $p^*=\frac{np}{n-p}$. One can show that $\Omega'$ is a $(q,p)$-Sobolev-Poincar\'e domain if 
\begin{equation}\label{eq:condition 2}
\alpha+\beta(n-1)-\frac{nq}{p^*}>0
\end{equation}
holds with $n=2$; see~\cite{bk95}. Note also that $\Omega$ is $\frac{1}{\alpha}$-John.

We next bend each mushroom $F_i$ to make it spiralling so that  the resulting domain $\Omega$ is an $s$-diam John domain. According to our choice, $s=\frac{\mu}{\beta}$. One can check that if $\beta=\alpha\frac{p+(p-1)q}{q-p}$, then~\eqref{eq:condition 2} reduces to
\begin{align}\label{eq:condition3}
\frac{1}{\beta}<\frac{p^2}{(2-p)[p+(p-1)q]}.
\end{align}
Since $p<q<p^*$, $\frac{p^2}{(2-p)[p+(p-1)q]}>1$. For any $\beta$ satisfies~~\eqref{eq:condition3} and $\beta=\alpha\frac{p+(p-1)q}{q-p}$. It is easy to check that $\frac{1}{\alpha}>\frac{\mu}{\beta}=s$. It is clear that $\Omega'$ and $\Omega$ are bi-Lipschitz equivalent and so the $(q,p)$-Sobolev-Poincar\'e inequality holds in $\Omega$ as well. Moreover, $\Omega$ satisfies all the required properties.

One could also modify the above example to the fractional $(q,p)$-Sobolev-Poincar\'e case, but the computations will be too complicated and so we omit it in the present paper.

\textbf{Acknowledgements}

The author wants to express his gratitude to Antti V.V\"ah\"akangas for posing the question in Jyv\"askyl\"a Analysis Seminar, which is the main motivation of the current paper, and for sharing the manuscript~\cite{div14}. The author also wants to thank Academy Professor Pekka Koskela for helpful discussions. 

\bibliographystyle{amsplain}

\end{document}